\documentclass[letter,12pt]{amsart}
\usepackage{wrapfig}
\usepackage{amsfonts}
\usepackage{amssymb}
\usepackage{amsmath,tabu}
\usepackage{amsthm}
\usepackage{comment}
\usepackage{graphics}
\usepackage{epstopdf}
\usepackage{pst-all,psfrag}
\usepackage[unicode]{hyperref}
\usepackage{mathtools}
\usepackage{subcaption}
\usepackage{multicol}

\usepackage{verbatim}
\usepackage{eurosym}
\usepackage[utf8]{inputenc}
\usepackage[english]{babel}
\usepackage[margin=1in]{geometry}
\usepackage{apptools}
\usepackage{tabularx}
\usepackage{enumitem}
\usepackage{bbm}

\newtheorem{theorem}{Theorem}
\newtheorem{proposition}[theorem]{Proposition}

\newtheorem{lemma}[theorem]{Lemma}

\newtheorem{corollary}[theorem]{Corollary}
\newtheorem{remark}[theorem]{Remark}

\newtheorem{definition}[theorem]{Definition}

\newtheorem{examples}[theorem]{Examples}

\numberwithin{theorem}{section}

\newcommand{\eps}{\varepsilon}

\newcommand{\ii}{\mathbf i}
\newcommand{\dd}{\mathrm d}

\renewcommand{\P}{\mathcal P}
\newcommand{\D}{\mathcal D}

\begin{document}
\title{Eigenvalues of Heckman-Polychronakos operators}

\author{Charles Dunkl}

\address[Charles Dunkl]{University of Virginia}
\email{cfd5z@virginia.edu}

\author{Vadim Gorin}

\address[Vadim Gorin]{University of California at Berkeley}
\email{vadicgor@gmail.com}

\date{August 9, 2025}

\maketitle

\begin{abstract}
  Heckman-Polychronakos operators form a prominent family of commuting differential-difference operators defined in terms of the Dunkl operators $\D_i$ as $\P_m= \sum_{i=1}^N (x_i \D_i)^m$. They have been known since 1990s in connection with trigonometric Calogero-Moser-Sutherland Hamiltonian and Jack symmetric polynomials. We explicitly compute the eigenvalues of these operators for symmetric and skew-symmetric eigenfunctions, as well as partial sums of eigenvalues for general polynomial eigenfunctions.
\end{abstract}

%\textcolor{blue}{[Cherednik says that his 1994 Advances paper might be relevant -- see (2.16), (2.19), (2.23) there.]}

\section{Introduction}

Our paper is motivated by two directions. On one hand, the theory of \emph{quantum integrable systems}, or more generally \emph{spectral theory}, seeks to construct families of commuting operators and to understand the structure of their eigenvalues and eigenfunctions. On the other hand, \emph{integrable probability} often exploits such operators to extract information about stochastic systems and their asymptotic behavior.

We investigate a family of differential-difference operators acting on polynomials in $N$ variables $x_1,\dots,x_N$, known as the \emph{Heckman–Polychronakos operators}, which were first introduced in \cite{heckman1991elementary,polychronakos1992exchange} in connection with the study of the (trigonometric) Calogero–Moser–Sutherland Hamiltonian. The key properties of interest for these operators were that they commute with each other and with the Calogero-Moser-Sutherland Hamiltonian, and that Jack symmetric polynomials (whose extensive theory was developed just a couple of years before that, see \cite{stanley1989some}, \cite[Chapter VI]{macdonald1998symmetric}) are their eigenfunctions.

Shortly after their introduction, two important developments occurred. First, \cite{cherednik1991unification,cherednik1994integration} introduced an alternative family of \emph{Cherednik operators} satisfying the same three properties. Second, the theory was lifted to the level of Macdonald polynomials, of which Jack polynomials are a degeneration (cf.\ \cite[Chapter 6]{macdonald1998symmetric}). In the following years Cherednik operators and their Macdonald generalizations played central roles in many developments, while the Heckman-Polychronakos operators received significantly less attention (cf.\ \cite{chalykh2024dunkl} for a recent historical overview). As a result, even the most basic properties of the latter operators, such as explicit formulas for their eigenvalues, remain poorly understood. The goal of this paper is to help fill this gap in the literature and to reinitiate the study of the Heckman–Polychronakos operators.

Some relationships between Cherednik and Heckman–Polychronakos operators are discussed in \cite[Section 5]{sergeev2015jack}. Both families of operators can be constructed via Newton power sums of certain "basic" operators. However, these basic operators exhibit different properties in the two theories: in Cherednik theory, the basic operators commute with each other, but their definition is not compatible with the natural action of the symmetric group $S_N$ on $(x_1, \dots, x_N)$; in Heckman–Polychronakos theory, the basic operators are defined in an $S_N$-equivariant way, but at the cost of losing commutativity.\footnote{We emphasize that although the basic operators $x_i\D_i$ do not commute, the Heckman–Polychronakos operators themselves do commute; see Definition \ref{Definition_HP_operator}, Lemma \ref{Lemma_commute}, and \eqref{eq_commutator}.} The two parallel choices of basic operators in Macdonald $(q,t)$ theory are further discussed in \cite{nazarov2019cherednik}, especially in Section~2.4.

In the semiclassical limit,\footnote{This limit is related to the degeneration of a Lie group into a Lie algebra and to the elementary transition $\lim\limits_{\eps \to 0}(1+\eps x)^{\eps^{-1}y} = \exp(xy)$.} Jack polynomials degenerate into multivariate Bessel functions (see, e.g., \cite[Section 4]{okounkov1997shifted}), and the operators we study degenerate in parallel to objects related to the rational (rather than trigonometric) Calogero–Moser–Sutherland Hamiltonian. Interestingly, in this limit the Cherednik and Heckman–Polychronakos operators merge into a single family: their basic operators converge to the Dunkl operators of \cite{dunkl1989differential}, which both commute and are $S_N$-equivariant.

Switching to the integrable probability side, recent works \cite{benaych2022matrix,xu2023rectangular,GXZ2024,KX2024} have made intensive use of power sums of Dunkl operators in the study of $\beta$-ensembles in random matrix theory. These studies address problems such as the addition of matrices, corner cutting, and the analysis of classical random matrix ensembles. Typical questions in this area concern the asymptotic behavior of individual eigenvalues, as well as more intricate functions of eigenvalues that capture their joint distribution as the matrix size tends to infinity. For the operators to be useful in such asymptotic analyses, two features are particularly important. First, the operators should have reasonably simple combinatorial expressions, allowing their action on explicit functions to be tractable in asymptotic regimes. Second, the eigenvalues of the operators correspond to observables of the stochastic system of interest, and thus should also be reasonably explicit. The power sums of Dunkl operators satisfy both of these criteria.

The random matrix distributions discussed above have natural discrete analogues, extensively studied in \cite{kerov2000anisotropic, borodin2005z, gorin2015multilevel,dolega2016gaussian,borodin2017gaussian, dolega2019gaussian, guionnet2019rigidity,huang2021law,dimitrov2022asymptotics,cuenca2023universality,moll2023gaussian, dimitrov2024global,gorin2024dynamical}, and closely connected to Jack polynomials. Many questions posed for matrix $\beta$-ensembles have discrete counterparts, and similar techniques may be applied to analyze them. This, however, requires lifting the power sums of Dunkl operators to the discrete setting --- leading naturally to either Cherednik or Heckman–Polychronakos operators. It is not yet clear which of the two classes is better suited for asymptotic analysis, and ideally, one would like to explore both.

The desire to study asymptotic properties of these stochastic systems was the original motivation behind our investigation of Heckman–Polychronakos operators. However, their theory turned out to be far richer than expected, revealing many additional structural features and open questions that remain to be explored.

\medskip

The rest of the paper is organized as follows. In Section \ref{Section_HP_operators} we formally define the Heckman-Polychronakos operators. In Section \ref{Section_sym_eigenfunction} we present a formula for their eigenvalues on symmetric eigenfunctions. In Section \ref{Section_classification} we propose a classification of general polynomial eigenfunctions according to the degrees in their leading monomials and the type of representation of the symmetric group $S_N$ to which these monomials belong. In Section \ref{Section_sums} we relate partial sums of eigenvalues to values of the characters of the symmetric group and explain how this can be used for explicit computations of some of the eigenvalues. In Section \ref{Section_N3} we compute a three variable example highlighting that in some cases the eigenvalues are much more complicated. Finally, Section \ref{Section_appendix} is an appendix, where we recall various useful properties of the Dunkl operators.

\medskip

\noindent {\bf Acknowledgements.} We would like to thank Evgeni Dimitrov and Grigori Olshanski for very helpful discussions. We thank the referees for reading the manuscript and providing valuable suggestions. V.G.\ was partially supported by NSF grant DMS - 2246449.

\section{Heckman-Polychronakos operators}

\label{Section_HP_operators}

Throughout the paper we fix $N=1,2,\dots$ and use the ring $\mathbb C_N=\mathbb C[x_1,x_2\dots,x_N]$ of polynomials in $N$ variables.  We let $\theta$ be an auxiliary parameter which will either play the
role of a formal variable or a positive number throughout the paper.
\begin{definition} \label{Definition_Dunkl}
 For $1\le i\le N$ the \emph{Dunkl operator} $\D_i$ is a linear operator in $\mathbb C_N$ acting by
 \begin{equation}
 \D_i= \frac{\partial}{\partial x_i}+\theta \sum_{j\ne i} \frac{1-(i,j)}{x_i-x_j},
 \end{equation}
 where $(i,j)$ stands for the operator permuting the variables $x_i$ and $x_j$.
\end{definition}
\begin{definition} \label{Definition_HP_operator}
 For $m\ge 1$ the \emph{Heckman-Polychronakos operator} $\P_m$ is a linear operator in $\mathbb C_N$ acting by
 \begin{equation}
 \P_m= \sum_{i=1}^N (x_i \D_i)^m,
 \end{equation}
 where $x_i$ stands for the operator of multiplication by $x_i$.
\end{definition}

The operators $\P_m$ were first introduced and discussed in \cite{heckman1991elementary,polychronakos1992exchange}. The interest in them is based on three properties, whose proofs we recall in the Appendix:
\begin{itemize}
\item $\P_m$ commute with each other as operators in $\mathbb C_N$:  $\P_m \P_k=\P_k\P_m$.
\item When restricted onto the subring $\Lambda_N\subset \mathbb C_N$ of symmetric polynomials, the operator $\P_2$ is the (conjugated) Calogero-Moser-Sutherland Hamiltonian.
\item Symmetric Jack polynomials $J_\lambda(x_1,\dots,x_N;\, \theta)$, $\lambda_1\ge\lambda_2\ge \dots\ge \lambda_N\ge 0$, are eigenfunctions of all $\P_m$; see \cite{stanley1989some}, \cite[Chaprer VI]{macdonald1998symmetric} for general introduction to Jack polynomials and note that it uses the parameter $\alpha=1/\theta$.
\end{itemize}

In contrast to another famous family of operators satisfying these three properties --- power sums of the (Jack versions of) Cherednik operators \cite{cherednik1991unification,cherednik1994integration} --- the individual terms $(x_i \D_i)^m$ \emph{do not} commute.\footnote{For a fruitful study of commutators of $\D_i$ and $x_j$ operators see \cite{feigin2015dunkl}.} On the other hand, the advantage of Heckman-Polychronakos operators is that their definition is simpler than the Cherednik ones and they are equivariant with respect to the symmetric group action.

The aim of this paper is to study the spectrum of the operators $\P_m$: we would like to understand the eigenvalues corresponding to Jack polynomials $J_\lambda$, as well as those corresponding to other \emph{non-symmetric} eigenfunctions.

The $m=1$ operator is simpler than others:
$$
 \sum_{i=1}^N x_i \D_i = \sum_{i=1}^N x_i \frac{\partial}{\partial x_i}+\theta \sum_{1\le i<j\le N} \frac{x_i-x_j}{x_i-x_j} [1-(i,j)]=
  \sum_{i=1}^N x_i \frac{\partial}{\partial x_i} + \frac{ \theta N(N-1)}{2} - \theta \sum_{1\le i <j\le N} (i,j).
$$
The appearance of $\sum_{i<j} (i,j)$ operator hints on the relevance of the representation theory of symmetric group $S_N$ for the computations of eigenvalues and we will see that this is indeed the case.

\section{Symmetric eigenfunctions}
\label{Section_sym_eigenfunction}

In this section we compute the eigenvalues corresponding to the (symmetric) Jack polynomials. 

A partition of $n$ is a sequence of non-negative integers $\lambda=(\lambda_1\ge \lambda_2\ge \dots)$, such that $\sum_{i} \lambda_i=n$. One can write $\lambda\vdash n$ or $|\lambda|=n$ in this situation. The length of $\lambda$, denoted $\ell(\lambda)$ is the number of non-zero parts in $\lambda$. In particular, if $\ell(\lambda)\le N$, then $\lambda_{N+1}=0$ and we can represent $\lambda$ as $N$-tuple $(\lambda_1\ge\lambda_2\ge \dots\ge\lambda_N\ge 0)$. Sometimes exponents are used to indicate multiplicity, e.g.\ $(3,1^2)=(3,1,1)$.

We also recall the notation $h_m(x_1,\dots,x_N)$ for complete homogeneous symmetric polynomials and introduce their relatives, see \cite[Example 19, Section I.2]{macdonald1998symmetric} for similar (but slightly different) polynomials:
\begin{equation}
\label{eq_x3}
 h^{(r)}_m(x_1,\dots,x_N)=\sum_{r-\text{element subsets }\{i_1,\dots,i_r\}\subset\{1,2,\dots,N\}}h_m(x_{i_1},\dots,x_{i_r}).
\end{equation}
For $m<0$ we set $h^{(r)}_m=0$.
Clearly,
$$h^{(1)}_m(x_1,\dots x_N)=x_1^m+x_2^m+\dots+x_N^m,$$
$$h^{(2)}_{m-1}(x_1,\dots,x_N)=\sum_{i<j} \frac{x_i^{m}-x_j^m}{x_i-x_j},$$
and
$$
 h^{(N)}_m(x_1,\dots,x_N)=h_m(x_1,\dots,x_N)=\sum_{j_1+j_2+\dots+j_N=m} x_1^{j_1} x_2^{j_2}\cdots x_N^{j_N}.
$$

\begin{theorem} \label{Theorem_symmetric_ef}
 Denote $\ell_i=\lambda_i+\theta(N-i)$, $1\le i \le N$. For each $\lambda_1\ge \dots\ge \lambda_N\ge 0$ we have
 \begin{equation}
  \P_m J_\lambda(x_1,\dots,x_N;\, \theta)= \mathrm{eig}_m(\lambda) J_\lambda(x_1,\dots_N;\, \theta),
 \end{equation}
 where
 \begin{equation}
 \label{eq_Jack_ev}
  \mathrm{eig}_m(\lambda)=h^{(1)}_m\bigl(\ell_1,\dots,\ell_N\bigr)-\theta h^{(2)}_{m-1}\bigl(\ell_1,\dots,\ell_N\bigr)+\dots+(-\theta)^{N-1}  h^{(N)}_{m+1-N}\bigl(\ell_1,\dots,\ell_N\bigr).
 \end{equation}
\end{theorem}

Theorem \ref{Theorem_symmetric_ef} is a combination of Theorem \ref{Theorem_Traces_general} and Lemma \ref{Lemma_symmetric_ev_match} below.
An equivalent compact expression for the eigenvalue is:
\begin{equation}
\label{eq_x7}
 \mathrm{eig}_m(\lambda)=\begin{pmatrix}
1 & 1 & 1 & \ldots & 1%
\end{pmatrix}%
\begin{pmatrix}
\ell_{1} & -\theta & -\theta & \ldots & -\theta\\
0 & \ell_{2} & -\theta & \ldots & -\theta\\
0 & 0 & \ell_{3} & \ldots & -\theta\\
\vdots & \vdots & \vdots & \ddots & \vdots\\
0 & 0 & 0 & \ldots &  \ell_{N}%
\end{pmatrix}
^{m}%
\begin{pmatrix}
1\\
1\\
1\\
\vdots\\
1
\end{pmatrix}.
\end{equation}
The eigenvalues $\mathrm{eig}_m(\lambda)$ also have a simple generating function:
\begin{proposition} \label{Proposition_sym_gen_funct} Denoting $\ell_i=\lambda_i+\theta(N-i)$, for each $\lambda_1\ge \dots\ge \lambda_N\ge 0$ we have:
\begin{equation}
\label{eq_x8}
 1-\theta z \sum_{m=0}^{\infty}  \mathrm{eig}_m(\lambda) z^m =\prod_{i=1}^N\left(1-\frac{\theta z}{1-\ell_i z}\right)=\prod_{i=1}^{N}\frac{1-(\ell_i+\theta)z}{1-\ell_i z}.
\end{equation}
\end{proposition}
\begin{proof}
 For each $1\le r\le N$, using the generating function for the complete homogeneous polynomials $h_m$ (see \cite[Chapter 1, Section 2]{macdonald1998symmetric}) we have
\begin{multline*}
 \sum_{m=0}^{\infty} h^{(r)}_{m+1-r}(\ell_1,\dots,\ell_N) z^m
 = z^{r-1} \sum_{r-\text{element subsets }\{i_1,\dots,i_r\}\subset\{1,2,\dots,N\}}\sum_{k=0}^{\infty} h_k(\ell_{i_1},\dots,\ell_{i_r}) z^k
 \\= z^{r-1} \sum_{r-\text{element subsets }\{i_1,\dots,i_r\}\subset\{1,2,\dots,N\}} \prod_{j=1}^r \frac{1}{1-\ell_{i_j}z}.
\end{multline*}
Hence, plugging into \eqref{eq_Jack_ev} we get
$$
  1-\theta z \sum_{m=0}^{\infty}  \mathrm{eig}_m(\lambda) z^m=
  \sum_{\begin{smallmatrix} \mathcal A\subset \{1,\dots,N\} \end{smallmatrix}}  (-\theta z)^{|\mathcal A|}
   \prod_{a\in \mathcal A} \frac{1}{1-\ell_{a}z}=\prod_{i=1}^N\left(1-\frac{\theta z}{1-\ell_i z}\right).\qedhere
$$
\end{proof}

\begin{remark}
 The fact that symmetric Jack polynomials are eigenfunctions of $\P_m$ is not new: there is no doubt that the authors of \cite{heckman1991elementary,polychronakos1992exchange,lapointe1996exact} were aware of it. However, the formulas for the eigenvalues are more elusive and we could not locate \eqref{eq_Jack_ev}, \eqref{eq_x7}, or \eqref{eq_x8} in the literature.

 G.~Olshanski pointed out to us that in the special case\footnote{At $\theta=1$ Jack polynomials turn into Schur polynomials, which are characters of the irreducible representations of the unitary group $U(N)$.} $\theta=1$, \eqref{eq_x7} matches the formula for the eigenvalue of the classical Casimir operator $C_m$, as given in \cite[(*) and Theorem 2 in Section 60]{zhelobenko1973compact}.

  Formula \eqref{eq_x8} is similar to the expression for eigenvalues of Nazarov-Sklyanin operators in the space of polynomials in infinitely many variables, see \cite[Section 6]{nazarov2013integrable}, and also \cite{mickler2023spectral} for further references and discussions related to these operators.
\end{remark}

\section{Classification of eigenfunctions}

\label{Section_classification}

Outside Jack polynomials, other eigenfunctions of $\P_m$ are much less understood. In this section, we propose a classification for them based on the degrees of their leading monomials and the isotype (equivalence class) of the representation of the symmetric group $S_N$ to which these monomials belong.

We need the  \emph{dominance order} on partitions:
\begin{definition}
 Let $\lambda=(\lambda_1\ge\dots\lambda_N\ge 0)$ and $\tilde \lambda=(\tilde \lambda_1\ge \dots\ge \tilde \lambda_N\ge 0)$ be two distinct partitions of length at most $N$. We write $\lambda\succ \tilde \lambda$, if  $\sum_{i=1}^N\lambda_i = \sum_{i=1}^N \tilde \lambda_i$ and for
  all $1\le k \le N$, we have $\sum_{i=1}^k \lambda_i \ge \sum_{i=1}^k \tilde \lambda_i$.
\end{definition}

Let $\gamma=(\gamma_1,\dots,\gamma_N)$ be a degree sequence of a monomial in $x_1,\dots,x_N$ and write $x^{\gamma}$ for $x_1^{\gamma_1} x_2^{\gamma_2}\cdots x_N^{\gamma_N}$. We let $\gamma^+$ denote the partition $\lambda_1\ge\dots\lambda_N\ge 0$ obtained by rearranging $\gamma_i$ in nonincreasing order.

\begin{definition} \label{Definition_V_lambda}
 For $\lambda=(\lambda_1\ge\dots\lambda_N\ge 0)$, we let $V_\lambda$ denote the linear space spanned by all monomials $x^{\gamma}$ with $\gamma^+=\lambda$.
\end{definition}

Note that $V_\lambda$ has a structure of a representation of symmetric group $S_N$ permuting the variables $x_i$. We write $(i,j)$ for the transpositions in the symmetric group and $(i,j,k\dots)$ for one-cycle permutations. When we write $\sigma_2 \sigma_1 f$ for a polynomial $f(x_1,\dots,x_N)$ and permutations $\sigma_1,\sigma_2\in S_N$, we mean that we first permute the variables according to $\sigma_1$ and then permute according to $\sigma_2$. For instance:
$$
 (1,2)\, x_1^2 x_2 = x_1 x_2^2,\qquad (1,3) (1,2)\, x_1^2 x_2 = (1,2,3)\, x_1^2 x_2= x_2^2 x_3.
$$

Let $V_{\lambda;\tau}$ denote the isotypical component of the irreducible representation $\tau$ in $V_\lambda$, so that
\begin{equation}
\label{eq_V_decomposition}
 V_\lambda=\bigoplus_{\tau} V_{\lambda;\tau}.
\end{equation}
We recall that irreducible representations of symmetric group $S_N$ are parameterized by partitions $\tau\vdash N$, see \cite[Section I.7]{macdonald1998symmetric}, \cite{okounkov1996new}, or \cite{james2006representation}  for the reviews of the representation theory of $S_N$. If we denote $\mathrm{Mult}(\lambda)$ the partition of $N$ representing multiplicities of coordinates in $\lambda$ (for each $a\ge 0$ we count how many times $a$ appears in $\lambda$, and then reorder decreasingly the resulting numbers), then the summation in \eqref{eq_V_decomposition} goes over $\tau$ such that $\mathrm{Mult}(\lambda)\preceq \tau$, see, e.g.\ \cite[Remark on page 115]{macdonald1998symmetric}.

Outside several special cases, for generic $\lambda$ and $\tau$ the numbers $\dim V_{\lambda;\tau}$ do not admit a closed formula, although they are of interest in algebraic combinatorics (being equivalent to \emph{Kostka numbers}, cf.\ \cite[Section I.6]{macdonald1998symmetric}). It is straightforward to show that the dependence of $\dim V_{\lambda;\tau}$ on $\lambda$ is only through $\mathrm{Mult}(\lambda)$.

\begin{theorem}\label{Theorem_classification_of_ef}
 For each triplet $(\lambda,\tau,i)$ with $\lambda=(\lambda_1\ge \dots\ge \lambda_N\ge 0)$, $\tau\vdash N$, and $1\le i \le \dim V_{\lambda;\tau}$, the operators $\P_m$, $m\ge 1$, have a joint eigenfunction $F_{\lambda,\tau,i}$, such that:
 \begin{enumerate}
  \item $F_{\lambda,\tau,i}$ is a homogeneous polynomial of degree $|\lambda|$.
  \item All monomials $x^{\gamma}$ with non-zero coefficients in $F_{\lambda,\tau,i}$ are such that $\gamma^+\preceq \lambda$.
  \item The part of $F_{\lambda,\tau,i}$ spanned by monomials in $V_\lambda$ belongs to $V_{\lambda;\tau}$.
 \end{enumerate}
 The polynomials $F_{\lambda,\tau,i}$ span $\mathbb C_N$ and there are no other polynomial eigenfunctions.
\end{theorem}
\begin{remark}
 We do not propose any explicit form for the dependence of $F_{\lambda,\tau,i}$ on $i$. One reason is that the combination of operators $\P_m$ and the structure of $S_N$-representation is not sufficient to distinguish them. Indeed, for any coefficients $\alpha$ and $\beta$ and any $g\in S_N$, the polynomial $\alpha \cdot F_{\lambda,\tau,i} +\beta \cdot g \circ F_{\lambda,\tau,i}$ is an eigenfunction of $\P_m$ satisfying the same three properties and with the same eigenvalue.
\end{remark}
\begin{examples}
 For each $\lambda$, one can choose $\tau$ to be a one-part partition, $\tau=(N)$, corresponding to the trivial representation of $S_N$. Then $V_{\lambda;\tau}$ is one dimensional space spanned by the sum of all monomials in $V_\lambda$. The corresponding eigenfunction is the symmetric Jack polynomial.

 For each $\lambda$ with distinct parts, one can choose $\tau$ to be a unique $N$-parts partition $\tau=1^{N}$, corresponding to the sign representation of $S_N$. Then $V_{\lambda;\tau}$ is one dimensional space spanned by the signed sum of all monomials in $V_\lambda$, i.e.\ $\sum_{\sigma\in S_N} (-1)^{\sigma}  \circ x_{\sigma(1)}^{\lambda_1} x_{\sigma(2)}^{\lambda_2}\cdots x_{\sigma(N)}^{\lambda_N}$. The corresponding eigenfunction is a skew-symmetric polynomial, which can be expressed as alternating sum of all non-symmetric Jack polynomials labeled by $N!$ compositions obtained by permuting the coordinates of $\lambda$ or, alternatively, as the symmetric Jack polynomial with shifted $\theta$ and multiplied by the $\prod_{i<j}(x_i-x_j)$, see \cite[(2.40)]{baker1997calogero}.

 For $\lambda=(a,a,\dots,a)$, the space $V_\lambda$ is one-dimensional and only $\tau=(N)$ is possible.

 For $\lambda$ with distinct parts, the space $V_\lambda$ is isomorphic to the regular representation of $S_N$ of dimension $N!$, because there is a bijection between elements of $S_N$ and monomials $x^\gamma$ with $\gamma^+=\lambda$. For each partition $\tau\vdash N$, $\dim V_{\lambda;\tau}$ is the square of the dimension\footnote{The dimension can be computed, e.g., by the hook length formula of \cite{frame1954hook}.} of the irreducible representation $\tau$, and this is the number of eigenfunctions $F_{\lambda,\tau,i}$.
\end{examples}

In the rest of this section we prove Theorem \ref{Theorem_classification_of_ef}. Two ingredients of the proof are triangularity of Lemma \ref{Lemma_triangularity} and self-adjointness of Lemma \ref{Lemma_self_adjoint}.

Let us introduce a linear operator $T_i$, which acts in $\mathbb C_N$ and leaves each $V_\lambda$ invariant:
\begin{equation}
\label{eq_x4}
 T_i x^{\gamma}= \left(\gamma_i + \theta \#\{j\mid \gamma_j<\gamma_i\}\right)x^{\gamma}- \theta \sum_{j\mid \gamma_j>\gamma_i} (i,j) x^{\gamma}, \qquad 1\le i \le N.
\end{equation}

\begin{lemma} \label{Lemma_triangularity}
 For any partition $\lambda$ and any degree sequence $\gamma$ with $\gamma^+=\lambda$, we have
 $$
 \P_m[x^{\gamma}]=(T_1^m+T_2^m+\dots+T_N^m)[x^{\gamma}]+(\text{linear combination of monomials } x^{\tilde \gamma}\text{ with }\tilde \gamma^+\prec\lambda).
 $$
\end{lemma}
\begin{proof}
 Note that
 \begin{multline}
   x_i \frac{1-(i,j)}{x_i-x_j} \bigl[x_1^{\gamma_1},\dots, x_N^{\gamma_N}\bigr]= \prod_{a\ne i,j} x_a^{\gamma_a} \cdot  x_i \frac{x_i^{\gamma_i} x_j^{\gamma_j}-x_i^{\gamma_j} x_j^{\gamma_i}}{x_i-x_j}\\
   = \prod_{a\ne i,j} x_a^{\gamma_a}\cdot \begin{cases} x_i^{\gamma_i} x_j^{\gamma_j}+ x_i^{\gamma_i-1} x_j^{\gamma_j+1}+\dots+ x_i^{\gamma_j+1} x_j^{\gamma_i-1},& \gamma_j<\gamma_i,\\
   -x_i^{\gamma_j} x_j^{\gamma_i}- x_i^{\gamma_j-1} x_j^{\gamma_i+1}-\dots- x_i^{\gamma_i+1} x_j^{\gamma_j-1},& \gamma_j>\gamma_i,\\
   0,& \gamma_j=\gamma_i. \end{cases}
 \end{multline}
Hence,
$$
 x_i \D_i [x^{\gamma}]= T_i x^{\gamma}+ (\text{linear combination of monomials } x^{\tilde \gamma}\text{ with }\tilde \gamma^+\prec\gamma^+).
$$
Iterating the last identity for $\P_m=\sum_{i=1}^N (x_i \D_i)^m$, we arrive at the statement of the lemma.
\end{proof}

\begin{lemma} \label{Lemma_self_adjoint}
 Consider a scalar product on $\mathbb C_N$, where $\overline{g}$ stands for complex conjugation and $\theta\ge 0$:
 \begin{equation} \label{eq_scalar_product}
  \langle f,g \rangle =\int_{0}^{2\pi}\cdots\int_{0}^{2\pi} f\bigl(e^{\ii \phi_1},\dots,e^{\ii\phi_N}\bigr) \overline{g\bigl(e^{\ii \phi_1},\dots,e^{\ii\phi_N}\bigr)}
  \prod_{1\le i<j \le N} \bigl|e^{\ii \phi_i} - e^{\ii \phi_j}\bigr|^{2\theta} \, \dd \phi_1\dots \dd\phi_N.
 \end{equation}
 Then the operators $x_i \D_i$, $1\le i \le N$, are self-adjoint with respect to it: for any $f,g,\in\mathbb C_N$
 $$
  \langle x_i\D_i f,g\rangle=  \langle f,  x_i\D_i g\rangle.
 $$
\end{lemma}
\begin{proof} Versions of this statement can be found in \cite{dunkl1991integral} and \cite{heckman1991elementary}. Changing the variables $x_i=\exp(-\ii \phi_i)$, the operator $x_i \D_i$ is transformed into
$$
 \ii \frac{\partial}{\partial \phi_i}+\theta \sum_{j\ne i} \frac{1-(i,j)}{1-e^{\ii (\phi_j-\phi_i)}},
$$
and we need to check that it is self-adjoint with respect to the scalar product
 \begin{equation} \label{eq_scalar_product}
  \langle f,g \rangle_{\phi} =\int_{0}^{2\pi}\cdots\int_{0}^{2\pi} f\bigl(\phi_1,\dots,\phi_N\bigr) \overline{g\bigl( \phi_1,\dots,\phi_N\bigr)}
  \prod_{1\le i<j \le N} \bigl|e^{\ii \phi_i} - e^{\ii \phi_j}\bigr|^{2\theta} \, \dd \phi_1\dots \dd\phi_N.
 \end{equation}
We first check that the operator  $\frac{(i,j)}{1-\exp(\ii (\phi_j-\phi_i))}$ is self-adjoint, which is the identity
\begin{multline*}
 \int_{0}^{2\pi}\cdots\int_{0}^{2\pi} \frac{\bigl[(i,j) f\bigl(\phi_1,\dots,\phi_N\bigr)\bigr] \overline{g\bigl( \phi_1,\dots,\phi_N\bigr)}}{1-e^{\ii (\phi_j-\phi_i)}}
  \prod_{1\le i<j \le N} \bigl|e^{\ii \phi_i} - e^{\ii \phi_j}\bigr|^{2\theta} \, \dd \phi_1\dots \dd\phi_N
  \\ \stackrel{?}{=}  \int_{0}^{2\pi}\cdots\int_{0}^{2\pi} \frac{f\bigl(\phi_1,\dots,\phi_N\bigr) \bigl[(i,j)\overline{g\bigl( \phi_1,\dots,\phi_N \bigr)}\bigr]}{1-e^{-\ii (\phi_j-\phi_i)}}
  \prod_{1\le i<j \le N} \bigl|e^{\ii \phi_i} - e^{\ii \phi_j}\bigr|^{2\theta} \, \dd \phi_1\dots \dd\phi_N,
\end{multline*}
that is tautologically true by renaming $\phi_i \leftrightarrow \phi_j$ in the integral. It remains to check that $\ii \frac{\partial}{\partial \phi_i}+\theta \sum_{j\ne i} \frac{1}{1-\exp(\ii (\phi_j-\phi_i))}$ is self-adjoint. For that we note that if $F(\phi)  $ is periodic and differentiable on the torus $0\leq\phi<2\pi$, then $\int_{0}^{2\pi}\frac{\partial}{\partial\phi}F\left(
\phi\right)  \mathrm{d}\phi=F\left(  2\pi\right)  -F\left(  0\right)  =0$. Hence, recording the fact that the $\frac{\partial}{\partial \phi_i}$ derivative of the integrand in the right-hand side of \eqref{eq_scalar_product} integrates to zero we obtain:
\begin{multline*}
 0=\int_{0}^{2\pi}\cdots\int_{0}^{2\pi} \left[  \frac{\partial}{\partial \phi_i} f \overline{g} +  f \frac{\partial}{\partial \phi_i} \overline{g} + \theta f\overline{g} \sum_{j\ne i} \frac{\ii e^{\ii \phi_i}}{e^{\ii \phi_i}-e^{\ii\phi_j}} + \theta f\overline{g} \sum_{j\ne i}  \frac{-\ii e^{-\ii \phi_i}}{e^{-\ii \phi_i}-e^{-\ii\phi_j}} \right]
   \\ \times \prod_{1\le i<j \le N} \bigl(e^{\ii \phi_i} - e^{\ii \phi_j}\bigr)^{\theta} \bigl(e^{-\ii \phi_i} - e^{-\ii \phi_j}\bigr)^{\theta} \, \dd \phi_1\dots \dd\phi_N.
\end{multline*}
 Multiplying by $\ii$ we get the desired self-adjointness statement.
\end{proof}
\begin{remark}
 Following \cite[Section 3]{dunkl1991integral} and \cite[Section 3]{dunkl2008reflection}, there is another scalar product,
 which also makes $x_i \D_i$ self-adjoint. This other scalar product is defined through
 $\langle f,g\rangle=f(\D_1,\dots,\D_N) g(x_1,\dots,x_N)\bigr|_{x_1=\dots=x_N=0}$ and it makes $x_i$ the adjoint operator to $\D_i$.
\end{remark}

\begin{proof}[Proof of Theorem \ref{Theorem_classification_of_ef}] The operators $\P_m$ are self-adjoint by Lemma \ref{Lemma_self_adjoint}, commutative by Lemma \ref{Lemma_commute} in the appendix, and essentially finite-dimensional because $\P_m$ preserves the space of degree $k$ polynomials for each $k$. Hence, they have a joint eigenbasis and we only need to show that the eigenfunctions satisfy the claimed properties.  Consider the space $\bigoplus\limits_{\mu\mid \mu\preceq \lambda} V_{\mu}$. By Lemma \ref{Lemma_triangularity}, this is an invariant space for $\P_m$. Because it is finite-dimensional and $\P_m$ is self-adjoint by Lemma \ref{Lemma_self_adjoint}, this space has a complete system of eigenfunctions of $\P_m$. Inductively counting eigenfunctions, which belong to such spaces with smaller $\lambda$, we conclude that in $\bigoplus\limits_{\mu\mid \mu\preceq \lambda} V_{\mu}$ there are precisely $\dim V_\lambda$ eigenfunctions, which have at least one non-zero coefficient among monomials from $V_\lambda$: these are various $F_{\lambda,\tau,i}$ with fixed $\lambda$ and varying $\tau$ and $i$. By triangularity of Lemma \ref{Lemma_triangularity}, for each such polynomial, its part spanned by the monomials from $V_\lambda$ is necessarily an eigenfunction of $(T_1^m+T_2^m+\dots+T_N^m)$. The latter operator has a symmetric expression in terms of $x_1,\dots,x_N$, and therefore it commutes with the action of $S_N$. Hence, by the Schur's lemma, each $V_{\lambda;\tau}$ is an invariant subspace for $(T_1^m+T_2^m+\dots+T_N^m)$. Therefore, there are precisely $\dim V_{\lambda;\tau}$ eigenfunctions of $(T_1^m+T_2^m+\dots+T_N^m)$ inside $V_{\lambda;\tau}$. By triangularity, each of them has a unique extension to an eigenfunction of $\P_m$ and these are $F_{\lambda,\tau,i}$, $1\le i \le \dim V_{\lambda;\tau}$.
\end{proof}

\section{Sums of eigenvalues}

\label{Section_sums}

In general, the eigenvalue of the operator $\P_m$ on the polynomial $F_{\lambda,\tau,i}$ can be expressed as a root of a polynomial equation (of degree at most $\dim(\tau)$) with coefficients being polynomials in $\lambda_i$ and $\theta$; therefore, this eigenvalue can be quite complicated, see, e.g.\ \eqref{eq_irrational_ev} in Section \ref{Section_N3}. We found that the sums of all eigenvalues corresponding to the same $\lambda$ and $\tau$, are much more manageable.

\subsection{Distinct degrees of variables}

For an irreducible representation of $S_N$ labeled by $\tau\vdash N$, we let $\dim \tau$ be its dimension  and let $\chi^{\tau}(g)$ be its character. In particular, $\dim(\tau)=\chi^{\tau}(\mathrm{Id})$. We recall the notation $h_m^{(r)}$ from \eqref{eq_x3}. We also recall that $(1,2,\dots,k)$ is the permutation having a single $k$-cycle and $N-k$ fixed points $\{k+1,\dots,N\}$.

\begin{theorem} \label{Theorem_Traces_distinct}
 Let $\mathrm{eig}_m(\lambda,\tau,i)$ denote the eigenvalue of $\P_m$ on the eigenfunction $F_{\lambda,\tau,i}$  of Theorem \ref{Theorem_classification_of_ef}. Suppose that $\lambda$ has all distinct parts. Then, setting $\ell_i=\lambda_i+\theta(N-i)$, we have
 \begin{equation}
 \label{eq_distinct_trace}
  \sum_{i=1}^{\dim V_{\lambda,\tau}} \mathrm{eig}_m(\lambda,\tau,i) = \dim \tau \sum_{k=1}^{\min(m+1,N)}  (-\theta)^{k-1}  h^{(k)}_{m+1-k}\bigl(\ell_1,\dots,\ell_N\bigr) \chi^\tau\bigl( (1,2,\dots,k) ).
 \end{equation}
\end{theorem}
When $\tau=(N)$, $\dim V_{\lambda,\tau}=1$ and the expression \eqref{eq_distinct_trace} turns into \eqref{eq_Jack_ev}; this eigenvalue corresponds to the symmetric eigenfunction. When $\tau=(1^N)$, also $\dim V_{\lambda,\tau}=1$ and the expression \eqref{eq_distinct_trace} turns into a formula similar to \eqref{eq_Jack_ev} but with $(-\theta)^{k-1}$ replaced with $\theta^{k-1}$; this eigenvalue corresponds to a skew-symmetric eigenfunction.

In general, the numbers $\chi^\tau\bigl( (1,2,\dots, k) )$ admit a somewhat explicit formula in terms of $\tau=(\tau_1\ge \dots\ge\tau_N\ge 0)$, see, e.g., \cite[Example I.7.7]{macdonald1998symmetric} for the following expression and \cite[Section 3]{ivanov2002kerov} for some others
\begin{equation}
\label{eq_one_cycle_characters}
 \chi^\tau\bigl( (1,2,\dots, k) )=\dim \tau \frac{(N-k)!}{N!} \sum_{i=1}^N \frac{(\tau_i+N-i)!}{(\tau_i+N-i-k)!} \prod_{j\ne i} \frac{\tau_i-i-\tau_j+j-k}{\tau_i-i-\tau_j+j}.
\end{equation}

If parts of $\lambda$ are allowed to coincide, \eqref{eq_distinct_trace} is replaced by a significantly more complicated formula, which we present in Theorem \ref{Theorem_Traces_general} below.

For the proof of Theorem \ref{Theorem_Traces_distinct} we use $T_i$ of \eqref{eq_x4}. We also need an additional object from the representation theory of $S_N$. For each irreducible representation $\tau$ of $S_N$, denote
\begin{equation}
\label{eq_projector}
 \pi_\tau= \frac{\dim \tau}{N!} \sum_{g\in S_N} \chi^\tau(g) g.
\end{equation}
\cite[Theorem 8]{serre1977linear} says that in each representation of $S_N$, $\pi_\tau$ acts as an orthogonal projector (with respect to an $S_N$-invariant scalar product) onto the isotypical component of $\tau$.

\begin{lemma} \label{Lemma_projector}
 The left-hand side of \eqref{eq_distinct_trace} can be evaluated as
 \begin{equation}
 \label{eq_eig_as_Trace}
   \sum_{i=1}^{\dim V_{\lambda,\tau}} \mathrm{eig}_m(\lambda,\tau,i) =\mathrm{Trace}_{V_\lambda} \left[\pi_\tau (T_1^m+\dots+T_N^m)\right]= N\,  \mathrm{Trace}_{V_\lambda} \left[\pi_\tau T_1^m\right].
 \end{equation}
\end{lemma}
\begin{proof}
 The second equality in \eqref{eq_eig_as_Trace} follows from the invariance of both the Trace and $\pi_\tau$ under conjugations with $g\in S_N$, and we only prove the first one.

 Recall from the proof of Theorem \ref{Theorem_classification_of_ef} that $\mathrm{eig}_m(\lambda,\tau,i)$ are eigenvalues of $(T_1^m+\dots+T_n^m)$, corresponding to eigenfunctions in $V_{\lambda;\tau}$. Hence, $\mathrm{eig}_m(\lambda,\tau,i)$ are eigenvalues of the restriction of $(T_1^m+\dots+T_n^m)$ onto $V_{\lambda;\tau}$, which is the same as eigenvalues of $\pi_\tau (T_1^m+\dots+T_n^m) \pi_\tau$. The sum of the eigenvalues coincides with the trace of the operator, and we get $\mathrm{Trace}_{V_\lambda} \left[\pi_\tau (T_1^m+\dots+T_n^m)\pi_\tau\right]$. Since $\mathrm{Trace}[AB]=\mathrm{Trace}[BA]$, we can move one $\pi_\tau$ inside the trace from right to the left, and then using the projecting property $\pi_\tau \pi_\tau=\pi_\tau$, we arrive at \eqref{eq_eig_as_Trace}.
\end{proof}

\begin{proof}[Proof of Theorem \ref{Theorem_Traces_distinct}] We use Lemma \ref{Lemma_projector} and compute $\mathrm{Trace}_{V_\lambda} \left[\pi_\tau T_1^m\right]$ combining \eqref{eq_x4} with \eqref{eq_projector}. The computation is based on the observation that $V_\lambda$ is isomorphic to the regular representation of $S_N$, because for distinct permutations $\sigma\in S_N$, their actions on $x_1^{\lambda_1}\cdots x_N^{\lambda_N}$ give distinct monomials. Therefore, in the trace computation all products of $g$ and $(1,j)$ not equal to the identity permutation do not give any contribution.

Let us compute $(T_1)^m x^{\gamma}$ for $\gamma=(\gamma_1,\gamma_2,\dots,\gamma_N)$. %The computation depends on the power of $x_1$, and let us assume that it is $d_r$, i.e.\ the $r$--th largest. Then
We have
\begin{equation}
\label{eq_x5}
 (T_1)^m x^{\gamma}=\sum_{k=1}^{m+1}\sum_{j_1,\dots,j_{k-1}} p_{k;q}[\gamma_1,\gamma_{j_1},\dots,\gamma_{j_{k-1}}]\cdot  (-\theta)^{k-1}(1,j_{k-1})(1, j_{k-2})\cdots (1,j_{1}) x^{\gamma},
\end{equation}
where $j_1,\dots,j_{k-1}$ are distinct indices, such that each of them corresponds to distinct $\gamma_{j_{k-1}}>\gamma_{j_{k-2}}>\dots>\gamma_{j_1}>\gamma_1$. The transpositions $(1,j_a)$ correspond to the second term in \eqref{eq_x4} and the ordering appears because the operator $T_1$ can only increase the degree of $x_1$, but never decrease. On the other hand, the prefactor $p_{k;q}$ collects all the factors coming from the first term; it is a polynomial in various $\theta$-shifted $\gamma_i$ of degree $(m+1-k)$, which depends on the collection $\gamma_1, \gamma_{j_1},\dots,\gamma_{j_{k-1}}$ and on $\theta$. Note also that the product of $k-1$ transpositions becomes a $k$-cycle
$$
 (1,j_{k-1})(1, j_{k-2})\cdots (1,j_{1})=(1, j_{1}, j_{2},\dots, j_{k-1}).
$$
We further represent $g$ in \eqref{eq_projector} as
$$
 g=\tilde g\cdot  (1,j_1)  \cdots (1,j_{k-2})(1, j_{k-1}),
$$
and note that when we plug into $\mathrm{Trace}_{V_\lambda} \left[\pi_\tau T_1^m\right]$, and multiply the sums, only the terms with $\tilde g=\mathrm{Id}$ give non-zero contributions.
We also note that the trace is a central function and its values on all $k$-cycles are the same. Therefore,
\begin{multline}
\label{eq_x6}
 N \mathrm{Trace}_{V_\lambda} \left[\pi_\tau T_1^m\right]= \frac{\dim \tau}{(N-1)!}  \sum_{k=1}^{m+1} (-\theta)^{k-1} \chi^\tau\bigl( (1,2,\dots,k) \bigr)
 \sum_{\gamma}  \sum_{j_1,\dots,j_{k-1}} p_{k;m}[\gamma_1,\gamma_{j_1},\dots,\gamma_{j_{k-1}}].
\end{multline}
It remains to compute the last double sum and we claim that
\begin{multline}
\label{eq_x11}
  \frac{1}{(N-1)!}\sum_{\gamma}  \sum_{j_1,\dots,j_{k-1}} p_{k;m}[\gamma_1,\gamma_{j_1},\dots,\gamma_{j_{k-1}}]\\=
  \text{Coefficient of }(-\theta)^{k-1}\text{ in: }
 \begin{pmatrix}1&1&\dots&1\end{pmatrix} \cdot
   \begin{pmatrix}
                 \ell_1& -\theta &\dots & && -\theta \\
                 0 & \ell_2& -\theta &&& -\theta\\
                 0 & 0 & \ell_3 & -\theta && \vdots\\
                 \vdots &\vdots &\vdots&&& -\theta\\
                 0 &0 &&\dots &0 &\ell_N
  \end{pmatrix}^m
  \cdot \begin{pmatrix} 1\\ 1 \\ \vdots \\ 1 \end{pmatrix}
\end{multline}
Indeed, there are $N!$ terms in the sum over $\gamma$, however, only the rank of $\gamma_1$ among all coordinates of $\gamma$ matters for the computation
of the following $ \sum_{j_1,\dots,j_{k-1}}$. Hence, $\frac{1}{(N-1)!}\sum_{\gamma}$ can be interpreted as a sum with $N$ terms corresponding to $N$ possible
ranks of $\gamma_1$.

Next, we notice that the first term in \eqref{eq_x4} is always equal to $\ell_a$ and therefore matches the diagonal element of the matrix
in \eqref{eq_x11}. Similarly, the $-\theta(i,j)$ terms in \eqref{eq_x4} match the off-diagonal $(-\theta)$ elements in the matrix  in \eqref{eq_x11}.
Hence, we see that the combinatorics of rising the matrix to the $m$-th power in \eqref{eq_x11} and raising $T_i$ to the $m$-th power is exactly
the same. Hence, the identity \eqref{eq_x11}.

It remains to identify the right-hand side of \eqref{eq_x11} with  $h^{(k)}_{m+1-k}\bigl(\ell_1,\dots,\ell_N\bigr)$, which is immediate from the definition \eqref{eq_x3}: the $k$-element subset $\{i_1,\dots,i_k\}$ encodes those diagonal elements of the matrix in \eqref{eq_x11}, which enter into the computation of the $m$-th power.

Finally, the summation $\sum_{k=1}^{m+1}$ in \eqref{eq_x6} can be restricted to $\sum_{k=1}^{\min(m+1,N)}$ in \eqref{eq_distinct_trace}, because for $k>N$ finding $k$ distinct coordinates $\gamma_1,\gamma_{j_1},\dots,\gamma_{j_{k-1}}$ becomes impossible and the corresponding sum in \eqref{eq_x6} is empty.
\end{proof}

\subsection{General degrees}

We now proceed to the most general result, in which parts of $\lambda$ are allowed to have multiplicities.

Given an integer $1\le p\le N$, a multiplicity composition (i.e.\ a sequence of $p$ positive integers) $\mathbf n=(n_1, n_2, \dots, n_p>0)$ with $\sum_{i=1}^p n_i=N$, and a degree sequence $\mathbf d=(d_1> d_2> \dots> d_p\ge 0)$, we introduce a partition
\begin{equation}
\label{eq_general_lambda}
\lambda= \underbrace{d_1,\dots, d_{1}}_{n_1}, \underbrace{d_2,\dots, d_{2}}_{n_2}, \dots, \underbrace{d_p,\dots, d_{p}}_{n_p}.
\end{equation}

For any $k$-element subset $\mathcal A\subset \{1,\dots,p\}$ of the form $\{a_1<a_2<\dots<a_k\}$, we introduce averaged characters also called ``spherical functions'':
\begin{equation}
\label{eq_family_character}
 \chi^\tau[\mathcal A; \mathbf n]=\frac{1}{(n_1)!(n_2)!\cdots (n_p)!}\sum_{\tilde g\in S_{n_1}\times S_{n_2}\times  \dots \times S_{n_p}}\,\, \chi^\tau(\tilde g \cdot \mathfrak c),
\end{equation}
where $\mathfrak c$ is a $k$-cycle, joining together the $S_{n_{a_i}}$, $1\le i\le k$, subgroups, in the order of increasing $i$. More explicitly:
$$
 \mathfrak c= (n_1+\dots+n_{a_1}-1,\, n_1+\dots+n_{a_2}-1,\, \dots, n_1+\dots+n_{a_{k-1}}-1,\, n_1+\dots+n_{a_{k}}-1).
$$
\begin{remark}
Due to centrality of the characters, one can also choose other numbers representing $S_{n_{a_i}}$ subgroups when defining $\mathfrak c$ and arrive at exactly the same function $\chi^{\tau}[\mathcal A; \mathbf n]$. For instance, $n_1+\dots+n_{a_1}-1$ can be replaced by any $j$, such that $n_1+\dots+ n_{a_1-1}<j \le  n_1+\dots+n_{a_1}-1$.
\end{remark}

We also introduce the shifted degrees:
$$
 \tilde \ell_i= d_i + \theta (n_{i+1}+n_{i+2}+\dots+n_p),\qquad 1\le i \le p.
$$
Note that $\{\tilde \ell_i\}_{i=1}^p$ is a subset of $\{\ell_i=\lambda_{i}+\theta(N-i)\}_{i=1}^N$. Finally, we define complete homogeneous polynomials in subsets of these numbers:
\begin{equation}
\label{eq_general_h}
 h^{\mathcal A}_{m}=h_m(\tilde \ell_{a_1},\tilde \ell_{a_2},\dots,\tilde \ell_{a_k}), \qquad \mathcal A=\{a_1,a_2,\dots,a_k\}.
\end{equation}
Note that for $p=N$, the polynomial $h^{(k)}_m$, as in \eqref{eq_x3}, is the sum of $h^{\mathcal A}$ over all $k$-element subsets.

\begin{theorem} \label{Theorem_Traces_general}
 Let $\mathrm{eig}_m(\lambda,\tau,i)$ denote the eigenvalue of $\P_m$ on the eigenfunction $F_{\lambda,\tau,i}$  of Theorem \ref{Theorem_classification_of_ef}. Suppose that $\lambda$ has the general form \eqref{eq_general_lambda}. Then, using the notations \eqref{eq_family_character} and \eqref{eq_general_h}, we have
 \begin{equation}
 \label{eq_general_trace}
  \sum_{i=1}^{\dim V_{\lambda,\tau}} \mathrm{eig}_m(\lambda,\tau,i) = \dim \tau \sum_{k=1}^{\min(m+1,p)}  (-\theta)^{k-1} \sum_{\begin{smallmatrix}\mathcal A\subset \{1,2\dots,p\}\\ |\mathcal A|=k\end{smallmatrix}}  h^{\mathcal A}_{m+1-k} \cdot   \chi^\tau[\mathcal A; \mathbf n]\cdot \prod_{a\in \mathcal A} n_a.
 \end{equation}
\end{theorem}

If $\lambda$ has distinct parts, i.e.\ $n_1=n_2=\dots=n_N=1$, then there is no averaging in \eqref{eq_family_character} and $ \chi^\tau[\mathcal A; \mathbf n]=\chi^\tau \bigl( (12\dots k) \bigr)$, where $k=|\mathcal A|$. Hence, Theorem \ref{Theorem_Traces_general} matches Theorem \ref{Theorem_Traces_distinct}.

If $\tau=(N)$ is the trivial representation, then the eigenfunction is the Jack polynomial (see Lemma \ref{Lemma_Symmetric_is_Jack} for some details) and simultaneously $\chi^\tau[\mathcal A; \mathbf n]=1$. Yet, the expression we get in \eqref{eq_general_trace} is slightly different from the one of Theorem \ref{Theorem_symmetric_ef}, and we prove their equivalence in Lemma \ref{Lemma_symmetric_ev_match} below.

If $\tau=(1^N)$ is the sign representation, then $\chi^{\tau}[\mathcal A; \mathbf n]=0$, unless $\lambda$ has all distinct parts. In fact, there is no $(1^N)$-type component in $V_\lambda$ for $\lambda$ with non-trivial multiplicities of parts.

Explicit evaluation of $\chi^{\tau}[\mathcal A; \mathbf n]$ for $\tau=(N-1,1)$ is later given in Theorem \ref{Theorem_character_sum_fundamental} and the $p=2$ case is explained in Corollary \ref{Corollary_ab}.

\begin{proof}[Proof of Theorem \ref{Theorem_Traces_general}]
The proof is an extended version of the argument in Theorem \ref{Theorem_Traces_distinct}. Let us compute $(T_1)^m x^{\gamma}$ for $\gamma^+=\lambda$, just as in \eqref{eq_x5}, but now allowing the coordinates of $\gamma_i$ to coincide. The answer has the same form:
\begin{equation}
 (T_1)^m x^{\gamma}=\sum_{k=1}^{m+1}\sum_{j_1,\dots,j_{k-1}} p_{k;m}[\gamma_1,\gamma_{j_1},\dots,\gamma_{j_{k-1}}]\cdot  (-\theta)^{k-1}(1,j_{k-1})(1, j_{k-2})\cdots (1,j_{1}) x^{\gamma},
\end{equation}
where $j_1,\dots,j_{k-1}$ are distinct indices, such that each of them corresponds to distinct $\gamma_{j_{k-1}}>\gamma_{j_{k-2}}>\dots>\gamma_{j_{1}}>\gamma_1$. We again represent $g$ as
$$
 g=\tilde g\cdot  (1,j_{1})  \cdots (1,j_{k-2})(1, j_{k-1}),
$$
and arrive at a generalization of \eqref{eq_x6}:
\begin{multline}
\label{eq_x16}
  \sum_{i=1}^{\dim V_{\lambda,\tau}} \mathrm{eig}_m(\lambda,\tau,i) =N \mathrm{Trace}_{V_{\lambda}} \Bigl( \pi_\tau (T_1)^m\Bigr)= \frac{\dim \tau}{(N-1)!}  \sum_{k=1}^{m+1} (-\theta)^{k-1}
 \\ \times \sum_{\gamma}  \sum_{j_1,\dots,j_{k-1}}  \sum_{\tilde g\in S_N\mid \tilde g\gamma=\gamma} \chi^\tau\bigl(\tilde g \cdot (1,j_{1})  \cdots (1,j_{k-2})(1, j_{k-1}) \bigr)  p_{k;m}[\gamma_1,\gamma_{j_1},\dots,\gamma_{j_{k-1}}]
\end{multline}
What remains is to carefully match the last formula with the claimed \eqref{eq_general_trace}. First, the summation index $k$ runs from $1$ to $m+1$ in \eqref{eq_x16}, however, for $k>p$, it is impossible to choose $k$ distinct $\gamma_{j_{k-1}}>\gamma_{j_{k-2}}>\dots>\gamma_{j_{1}}>\gamma_1$ and the sum is empty; hence, we can restrict the upper bound to be $\min(k+1,p)$.

Next, the summation over all $\gamma$ involves
$$
 \frac{N!}{(n_1)! (n_2)!\cdots (n_p)!}
$$
terms. Absorbing the factorials from the denominator into the character summation (they are present in \eqref{eq_family_character}) and cancelling $N!$ with $(N-1)!$ in denominator in \eqref{eq_x16}, we are left with a sum involving a single fixed $\gamma$ and $ \pi_\tau \sum_{i=1}^N (T_i)^m$.

We claim that the summation over $\tilde g$ in \eqref{eq_x16} gives the character averages of \eqref{eq_family_character}. Indeed, note that
$$
(1,j_{1})  \cdots (1,j_{k-2})(1, j_{k-1})=(j_{k-1}, j_{k-2},\dots,j_{1},1)
$$
and the corresponding degrees $\gamma_{j_{k-1}}, \gamma_{j_{k-2}},\dots,\gamma_{1}$ are arranged in the decreasing order, which correspond to increasing indices of $d_i$ --- matching the cycle $\mathfrak c$ used in \eqref{eq_family_character}. The stabilizer of $\gamma$, i.e.\ the subgroup $\{\tilde g\in S_N\mid \tilde g\gamma=\gamma\}$ is not necessarily $S_{n_1}\times S_{n_2}\times \dots S_{n_p}$, but it is isomorphic to this group, with isomorphism obtained by renaming $\{1,2,\dots,N\}$. Since the characters are central, this leads to the same function as \eqref{eq_family_character}.

It remains to identify the sums of $p_{k;m}[\gamma_1,\gamma_{j_1},\dots,\gamma_{j_{k-1}}]$. Note that when we act with $(T_i)^m$ using its definition \eqref{eq_x4}, the factors $\left(\gamma_i + \theta \#\{j\mid \gamma_j<\gamma_i\}\right)$ correspond to various $\tilde \ell_a$. On the other hand, each application
of a term from $\theta \sum_{j\mid \gamma_j>\gamma_i} (i,j) x^{\gamma}$, leads to addition of an additional index ${j_a}$. In this index, only the value of $\gamma_{j_a}$ matters. The set $\mathcal A$ in \eqref{eq_general_trace} is identified with $\{\gamma_1,\gamma_{j_1},\dots,\gamma_{j_{k-1}}\}$ in \eqref{eq_x16}, leading to the combinatorial factor $\prod_{a\in \mathcal A} n_a$, counting the number of ways to choose indices $\gamma_{j_a}$ corresponding to $\mathcal A$ (the $n_a$ corresponding to $\gamma_1$ should be also accounted, as it arises from the first application of one of the operators $T_i$ in $\sum_{i=1}^N (T_i)^m$). Once the set $\mathcal A$ is fixed, the sums of products of $\tilde \ell_a$ coming from $\left(\gamma_i + \theta \#\{j\mid \gamma_j<\gamma_i\}\right)$ in $T_i$ combine precisely into $h^{\mathcal A}_{m+1-k}$ of \eqref{eq_general_trace}.
\end{proof}

\begin{lemma} \label{Lemma_symmetric_ev_match} For the trivial character, $\tau=(N)$, the formulas of Theorems \ref{Theorem_Traces_general} and \ref{Theorem_symmetric_ef} are the same.
\end{lemma}
\begin{proof} The formula \eqref{eq_Jack_ev} can be rewritten in terms of the $m$-th power of $N\times N$ matrix as
\begin{equation}
\label{eq_x12}
\begin{pmatrix}
1 & 1 & 1 & \ldots & 1%
\end{pmatrix}%
\begin{pmatrix}
\ell_{1} & -\theta & -\theta & \ldots & -\theta\\
0 & \ell_{2} & -\theta & \ldots & -\theta\\
0 & 0 & \ell_{3} & \ldots & -\theta\\
\vdots & \vdots & \vdots & \ddots & \vdots\\
0 & 0 & 0 & \ldots &  \ell_{N}%
\end{pmatrix}
^{m}%
\begin{pmatrix}
1\\
1\\
1\\
\vdots\\
1
\end{pmatrix},
\end{equation}
where $\ell_i=\lambda_i+\theta(N-i)$. On the other hand, the formula \eqref{eq_general_trace} can rewritten in terms of $m$-th power of $p\times p$ matrix as
\begin{equation}
\label{eq_x13}
\begin{pmatrix}
n_{1} & n_{2} & n_{3} & \ldots & n_{p}%
\end{pmatrix}%
\begin{pmatrix}
\tilde \ell_{1} & -n_{2}\theta & -n_{3}\theta & \ldots & -n_{p}\theta\\
0 & \tilde \ell_{2} & -n_{3}\theta & \ldots & -n_{p}\theta\\
0 & 0 & \tilde \ell_{3} & \ldots & -n_{p}\theta\\
\vdots & \vdots & \vdots & \ddots & \vdots\\
0 & 0 & 0 & \ldots & \tilde \ell_{p}%
\end{pmatrix}
^{m}%
\begin{pmatrix}
1\\
1\\
1\\
\vdots\\
1
\end{pmatrix},
\end{equation}
where $\tilde \ell_i$ are a subset of $\ell_i$ given by $\tilde \ell_i= d_i + \theta (n_{i+1}+n_{i+2}+\dots+n_p)$
in the notation \eqref{eq_general_lambda}. Our task is to prove \eqref{eq_x12}=\eqref{eq_x13}. One can do this directly, but our approach is to proceed
through generating functions. On the \eqref{eq_x12} side, the generating function was computed in Proposition \ref{Proposition_sym_gen_funct}. Let $\widetilde{\mathrm{eig}}_m(\lambda)$ denote the eigenvalues computed by \eqref{eq_x13}. Repeating the argument of Proposition \ref{Proposition_sym_gen_funct}, we have
\begin{multline*}
 1-\theta z \sum_{m=0}^{\infty}  \widetilde{\mathrm{eig}}_m(\lambda) z^m =
  \sum_{\begin{smallmatrix} \mathcal A\subset \{1,\dots,p\} \end{smallmatrix}}  (-\theta z)^{|\mathcal A|}
   \prod_{a\in \mathcal A} \frac{n_a}{1-\tilde \ell_{a}z}=\prod_{i=1}^p\left(1-\frac{n_i \theta z}{1-\tilde \ell_i z}\right)\\=\prod_{i=1}^p\frac{1-(\tilde \ell_i+n_i \theta) z}{1-\tilde \ell_i z}.
\end{multline*}
The last expression matches the right-hand side of \eqref{eq_x8}, once we notice the telescoping cancellations in \eqref{eq_x8}.
\end{proof}

\subsection{$(N-1,1)$ isotype}

While the formula of Theorem \ref{Theorem_Traces_general} is very general, it produces an answer in terms of the functions $\chi^\tau[\mathcal A; \mathbf n]$. Evaluation of these functions is a separate task. As we have already mentioned right after Theorem \ref{Theorem_Traces_general}, for $\tau=(N)$ and $\tau=1^N$ the evaluation is particularly simple. Another case where it is fully explicit is $\tau=(N-1,1)$.

\begin{theorem}
 \label{Theorem_character_sum_fundamental} Let $\tau=(N-1,1)$ and  $\mathbf n=(n_1, n_2, \dots, n_p>0)$. Then we have
\begin{equation} \label{eq_N11_formula}
 \chi^{\tau}[\mathcal A; \mathbf n]=p-1-\sum_{a\in \mathcal A} \frac{1}{n_a}.
\end{equation}
\end{theorem}
\begin{proof} The irreducible representation of $S_N$ of type $\tau=(N-1,1)$ is the standard (reflection) $(N-1)$-dimensional representation in the space of vectors $(x_1,\dots,x_N)$ with $\sum_{i=1}^N x_i=0$ and action $\sigma\circ (x_1,\dots,x_N)=(x_{\sigma^{-1}(1)},\dots,x_{\sigma^{-1}(N)})$. Evaluating the character as trace, we get an expression in terms of the number of the fixed points:
$$
  \chi^{\tau}(\sigma)=\#\{1\le i \le N\mid \sigma(i)=i\}-1, \qquad \tau=(N-1,1).
$$
Recalling the definition \eqref{eq_family_character}, we need to compute
\begin{equation}
\label{eq_x18}
\sum_{\tilde g\in S_{n_1}\times S_{n_2}\times  \dots \times S_{n_p}}\,\, \chi^\tau(\tilde g \cdot \mathfrak c)
=\sum_{\tilde g} \sum_i \mathbf 1_{ \tilde g \cdot \mathfrak c (i)=i}  - \prod_{j=1}^p (n_j)!,
\end{equation}
where $\mathfrak c$ is a $|\mathcal A|$-cycle permuting the groups corresponding to the elements $a\in\mathcal A$. Changing the order of summation in \eqref{eq_x18}, we get
$$
 \sum_{q=1}^p\, \sum_{i=n_1+\dots+n_{q-1}+1}^{n_1+\dots+n_q} \, \sum_{\tilde g\in S_{n_1}\times S_{n_2}\times  \dots \times S_{n_p}} \mathbf 1_{ \tilde g \cdot \mathfrak c (i)=i}  - \prod_{j=1}^p (n_j)!.
$$
In the internal sum over $\tilde g$, for the symmetric groups $S_{n_j}$ which permute the elements other than $ \mathfrak c (i)$ there are no restrictions and all $(n_j)!$ permutations are possible. However, if $\mathfrak c(i)$ belongs to the indices permuted by $S_{n_j}$, then there are two options: if $\mathfrak c(i)=i$, then $(n_j-1)!$ permutations satisfy the restriction $ \tilde g \cdot \mathfrak c (i)=i$, otherwise no permutations satisfy the restriction. Hence, summing over $\tilde g$ and then $i$, we get
$$\sum_{q=1}^p \left[\mathbf 1_{q\in \mathcal A} (n_q-1)+\mathbf 1_{q\not\in\mathcal A} n_q \right] (n_q-1)! \left[\prod_{j\ne q} (n_j)!\right]
 - \prod_{j=1}^p (n_j)!.
$$
The last expression can be transformed into
$$
(p-1)\prod_{j=1}^p (n_j)! - \sum_{q=1}^p \mathbf 1_{q\in \mathcal A} (n_q-1)! \left[\prod_{j\ne q} (n_j)!\right].
$$
Dividing by $\prod\limits_{j=1}^p (n_j)!$, we get \eqref{eq_N11_formula}.
\end{proof}

\subsection{$a\dots ab\dots b$ case with $p=2$} Another situation where the general formula of Theorem \ref{Theorem_Traces_general} is fully explicit is $p=2$. In fact, in this special case we are able to compute all eigenvalues $\mathrm{eig}_m(\lambda,\tau,i)$ rather than only their sums over $i$. We fix $N=1,2,\dots$ and $1\le \eta \le N-1$ and set the multiplicities to be $(n_1,n_2)= (N-\eta,\eta)$. We further choose $a>b\ge 0$ and set
\begin{equation}
\label{eq_two_row}
 \lambda= \underbrace{a,a,\dots,a}_{N-\eta}, \underbrace{b,b,\dots,b}_{\eta}.
\end{equation}
\begin{corollary} \label{Corollary_ab}
 For $\lambda$ in \eqref{eq_two_row} and $\tau=(N-k,k)$ with $0\le k\le \min(\eta,N-\eta)$, we have
 $$
   \mathrm{eig}_m(\lambda,\tau,i) = (a+\theta \eta)^m (N-\eta) + b^m \eta-\theta \frac{(a+\theta \eta)^m- b^m}{a+\theta \eta -b}  \bigl(\eta (N-\eta) -k(N-k+1)\bigr) ,
 $$
 for all $1\le i\le \dim V_{\lambda,\tau}$. For other $\tau$, the space $V_{\lambda,\tau}$ is empty.
\end{corollary}
\begin{proof}
 We need some facts from the representation theory of symmetric groups, which can be found e.g.\ in \cite{dunkl1978addition}, or \cite[Section 3]{Stanton1984}, or \cite[Sections 6.1, 6.2]{ceccherini2008harmonic}. In particular, $V_\lambda$ expands into a direct multiplicity one sum of irreducible representations $\tau=(N-k,k)$, $0\le k\le \min(\eta,N-\eta)$.
 The operator $(T_1^m+\dots+T_N^M)$ commutes with the action of the symmetric group. Therefore, by Schur's lemma, it necessarily acts as a multiple of the identity in each $V_{\lambda;\tau}$ (which would not have been true if $\tau$ had multiplicity larger than $1$ in $\lambda$). Hence, the eigenvalues $\mathrm{eig}_m(\lambda,\tau,i)$ do not depend on $i$ and we can get them by dividing \eqref{eq_general_trace} by $\dim \tau$.

 Let us evaluate the ingredients of \eqref{eq_general_trace}. $(S_N,S_{N-\eta}\times S_\eta)$ is a Gelfand pair, which means that each irreducible representation of $S_N$ has at most one $S_{N-\eta}\times S_\eta$-invariant vector. Interpreting summations over $\tilde g \in S_{N-\eta}\times S_\eta$ in \eqref{eq_family_character} as projections on invariants, we identify $ \chi^\tau[\mathcal A; \mathbf n]$ with spherical functions corresponding to this invariant vector. The latter spherical function is evaluated in \cite[Corollary 2.2]{dunkl1978addition}, \cite[(3.15)]{Stanton1984},  \cite[Theorems 6.1.10, 6.2.3]{ceccherini2008harmonic} in terms of the hypergeometric function $_3F_2$, which can also be identified with the Hahn polynomial. The value at length $2$ cycle is given as
 \begin{equation}
  _3 F_2\left( \begin{smallmatrix}-k,\, k-N-1,\, -1 \\ \eta-N, -\eta  \end{smallmatrix} \Bigr|  1 \right)=1 + \frac{(-k) (k-N-1) (-1)}{(-\eta)(\eta-N) 1!}=1-\frac{k(N-k+1)}{\eta(N-\eta)}.
 \end{equation}
% The latter spherical
% function is computed in. We need its value on length $1$ cycle --- which is the dimension of the space of invariant vectors and therefore is $1$ --- and its value on length $2$ cycle which is given by $\phi(N,h,m,k;1)$ from that book with $h=m=\eta$. Copying the definition\footnote{As \cite[Remark 6.2.4]{ceccherini2008harmonic} explains, these can be identified with values of Hahn polynomials.}  from \cite[Theorem 6.2.3]{ceccherini2008harmonic}:
% $$
% \phi(N,\eta,\eta,k;\ell)=\frac{(-1)^k}{{{N-\eta}\choose {k}}} \sum_{i=\max(0,\ell-\eta+k)}^{\min(\ell,k)} {{\eta-\ell}\choose{k-i}} %{{\ell}\choose{i}} \frac{(N-\eta-k+1)_{k-i}}{(-\eta)_{k-i}}
% $$
%For $\ell=0$, this is $1$.
% Assuming $0<k<\eta$, for $\ell=1$, this becomes
% \begin{multline*}
% \phi(N,\eta,\eta,k;1)=\frac{(-1)^k}{{{N-\eta}\choose {k}}} \left[ {{\eta-1}\choose{k}}  \frac{(N-\eta-k+1)_{k}}{(-\eta)_{k}} %+{{\eta-1}\choose{k-1}} \frac{(N-\eta-k+1)_{k-1}}{(-\eta)_{k-1}} \right]
% \\=1-\frac{k(N-\eta)+k (\eta-k+1)}{\eta(N-\eta)}   \\=1-\frac{k}{\eta}\cdot \frac{N-k+1}{N-\eta }.
% \end{multline*}
% The last formula also works for the boundary cases $k=0$ and $k=\eta$.
Hence, \eqref{eq_general_trace} specializes to:
\begin{multline*}
   \sum_{\begin{smallmatrix}\mathcal A\subset \{1,2\}\\ |\mathcal A|=1\end{smallmatrix}}  h^{\mathcal A}_{m}  \cdot \prod_{\alpha\in \mathcal A} n_\alpha
 -\theta \sum_{\begin{smallmatrix}\mathcal A\subset \{1,2\}\\ |\mathcal A|=2\end{smallmatrix}}  h^{\mathcal A}_{m-1} \cdot  \left(1-\frac{k}{\eta}\cdot \frac{N-\eta+1}{N-\eta }\right)\cdot \prod_{\alpha\in \mathcal A} n_\alpha.
 \\=(a+\theta \eta)^m (N-\eta) + b^m \eta-\theta \frac{(a+\theta \eta)^m- b^m}{a+\theta \eta -b}  \left(1-\frac{k}{\eta}\cdot \frac{N-k+1}{N-\eta }\right) \eta (N-\eta). \qedhere
 \end{multline*}
\end{proof}

\subsection{More complicated isotypes}
In the follow-up work, the spherical functions $\chi\left[  \mathcal{A},\mathbf{n}\right]  $ have
been computed for the hook isotypes $\left[  N-b,1^{b}\right]  $ for $p>b$ in
\cite{Dun25a} and for the isotypes $\left[  N-k,k\right]  $ for $p=3$ in
\cite{Dun25b}.

\section{Example: Eigenvalues for $N=3$.}
\label{Section_N3}

Here is the full list of the eigenvalues of the operator $\P_m$ for $N=3$. We keep using the notation \eqref{eq_x3}. When we speak about leading monomials in the description of eigenfunctions, we use the point of view of Theorem \ref{Theorem_classification_of_ef} and describe the part of $F_{\lambda,\tau,i}$ in $V_\lambda$.

\begin{enumerate}[label=(\roman*)]
 \item $\lambda=(a,a,a)$: For each $a\ge 0,$ $\P_m$ has an eigenvalue $3 a^2 m$ on the eigenfunction $x_1^a x_2^a x_3^a$.
 \item $\lambda=(a,a,b)$: For each $a>b \ge 0$, $\P_m$ has:
  \begin{enumerate}
   \item Eigenvalue $2 (a+\theta)^m + b^m - 2 \theta\frac{(a+\theta)^m-b^m}{a+\theta-b}$ on the eigenfunction with leading monomials $x_1^a x_2^a x_3^b + x_1^a x_2^b x_3^a+ x_1^b x_2^a x_3^a$.
   \item Two equal eigenvalues $2 (a+\theta)^m + b^m + \theta\frac{(a+\theta)^m-b^m}{a+\theta-b}$ on two eigenfunctions with leading monomials in the space spanned by $x_1^a x_2^a x_3^b - x_1^a x_2^b x_3^a$ and $x_1^a x_2^a x_3^b - x_1^b x_2^a x_3^a$.
   \end{enumerate}
 \item $\lambda=(a,b,b)$: For each $a>b \ge 0$, $\P_m$ has:
   \begin{enumerate}
    \item Eigenvalue  $(a+2\theta)^m + 2b^m - 2 \theta\frac{(a+2\theta)^m-b^m}{a+2\theta-b}$  on the eigenfunction with leading monomials $x_1^a x_2^b x_3^b + x_1^b x_2^a x_3^b+ x_1^b x_2^b x_3^a$.
    \item Two equal eigenvalues $(a+2\theta)^m + 2b^m + \theta\frac{(a+2\theta)^m-b^m}{a+2\theta-b}$  on two eigenfunctions with leading monomials in the space spanned by $x_1^a x_2^b x_3^b - x_1^b x_2^a x_3^b$ and $x_1^a x_2^b x_3^b - x_1^b x_2^b x_3^a$.
   \end{enumerate}
 \item $\lambda=(a,b,c)$: For each $a>b>c \ge 0$, $\P_m$ has (we denote $(\ell_1,\ell_2,\ell_3)=(a+2\theta,b+\theta,c)$):
   \begin{enumerate}
    \item Eigenvalue $h^{(1)}_m\bigl(\ell_1,\ell_2,\ell_3\bigr)-\theta h^{(2)}_{m-1}\bigl(\ell_1,\ell_2,\ell_3\bigr)+\theta^2  h^{(3)}_{m-2}\bigl(\ell_1,\ell_2,\ell_3\bigr)$ on the eigenfunction with leading monomials
    $$x_1^a x_2^b x_3^c +x_1^b x_2^a x_3^c+x_1^a x_2^c x_3^b+x_1^c x_2^b x_3^a+x_1^b x_2^c x_3^a+x_1^c x_2^a x_3^b.$$
    \item Two pairs of equal eigenvalues
    \begin{multline} \label{eq_irrational_ev}
      \left(  \ell_{1}^{m}+\ell_{2}^{m}+\ell_{3}^{m}\right)
      -\frac{\theta^{2}}{2}h_{m-2}\left(  \ell_{1},\ell_{2},\ell_{3}\right) \\  \pm\frac{\theta}{2}
      h_{m-2}\left(  \ell_{1},\ell_{2},\ell_{3}\right)  \sqrt{4\left(  \ell_{1}^{2}+\ell_{2}^{2}+\ell_{3}^{2}\right)  -4\left(  \ell_{1}\ell_{2}+\ell_{1}\ell_{3}+\ell_{2}\ell_{3}\right)  -3\theta^{2}}
   \end{multline}
     on four eigenfunctions with leading monomials in the space spanned by
    \begin{equation}
    \label{eq_x17}
    x_1^a x_2^b x_3^c-x_1^b x_2^c x_3^a,\quad x_1^a x_2^b x_3^c-x_1^c x_2^a x_3^b,\quad x_1^b x_2^a x_3^c-x_1^a x_2^c x_3^b,\quad x_1^b x_2^a x_3^c-x_1^c x_2^b x_3^a.
    \end{equation}
    \item Eigenvalue $h^{(1)}_m\bigl(\ell_1,\ell_2,\ell_3\bigr)+\theta h^{(2)}_{m-1}\bigl(\ell_1,\ell_2,\ell_3\bigr)+\theta^2  h^{(3)}_{m-2}\bigl(\ell_1,\ell_2,\ell_3\bigr)$ on the eigenfunction with leading monomials
     $$x_1^a x_2^b x_3^c -x_1^b x_2^a x_3^c-x_1^a x_2^c x_3^b-x_1^c x_2^b x_3^a+x_1^b x_2^c x_3^a+x_1^c x_2^a x_3^b.$$
   \end{enumerate}
\end{enumerate}

\medskip

(i.a), (ii.a), (iii.a), and (iv.a) are all instances of Theorem \ref{Theorem_symmetric_ef} or Theorem \ref{Theorem_Traces_general} with $\tau=(3)$. (ii.ab) and (iii.ab) are instances of Corollary \ref{Corollary_ab}. (iv.a) and (iv.c) are instances of Theorem \ref{Theorem_Traces_distinct} for $\tau=(3)$ and $\tau=(1,1,1)$, respectively. The sum of all four eigenvalues in (iv.b) matches Theorem \ref{Theorem_Traces_distinct} for $\tau=(2,1)$, where we have $\dim\tau=\chi^{\tau}(\mathrm{Id})=2$, $\chi^{\tau}\bigl( (1,2)\bigr)=0$, $\chi^{\tau}\bigl( (1,2,3)\bigr)=-1$. The individual eigenvalues in (iv.b) are not computed by any of our theorems; the formula \eqref{eq_irrational_ev} can be obtained by evaluating the matrices of the operators $T_1^m$, $T_2^m$, $T_3^m$ by diagonalization and then direct computation of the eigenvalues of $T_1^m+T_2^m+T_3^m$ (we can either use $6\times 6$ matrices in the space of all permutations of $x_1^a x_2^b x_3^c$ or $4\times 4$ matrices in the subspace spanned by \eqref{eq_x17}).

% In 6-dim space the trace of $(1,2,3)$ is 0. For symmetric part it is $1$. For skew-cymmetric part it is also $1$. So for the remaining $4$--dimensional space it is $-2$. This has two copies of irreducible representation $\tau$. Hence, for each of them the trace is $-1$.

\section{Appendix: Basic properties of $\P_m$}
\label{Section_appendix}

In this section we discuss the basic properties of the operators $\P_m$, mentioned in Section \ref{Section_HP_operators}. These properties are well-known, we provide proofs only for the sake of being self-contained.

\begin{lemma} \label{Lemma_commute}
 The operators $\P_m$, $m=1,2,\dots$, mutually commute.
\end{lemma}
\begin{proof}
 We develop a table of commutators. We start by splitting the Dunkl operator of Definition \ref{Definition_Dunkl} into two parts and write $\D_i=\frac{\partial}{\partial x_i} + \theta \Delta_i$ with $\Delta_i=\sum_{j\ne i} \frac{1-(i,j)}{x_i-x_j}$.
 We claim
 \begin{equation}
 \label{eq_Delta_commute}
  \Delta_a \Delta_b = \Delta_b \Delta_a,\qquad 1\le a,b \le N.
 \end{equation}
 In order to prove \eqref{eq_Delta_commute} for $a\ne b$, we plug the definition of $\Delta_a$ into $\Delta_a\Delta_b$ and get
 $$
  \sum_{j\ne a} \sum_{j'\ne b} \frac{1-(a,j)}{x_a-x_j} \cdot \frac{1-(b,{j'})}{x_b-x_{j'}}.
 $$
 Note that the double sum splits into three types of summands: involving two, three, and four distinct variables $x_i$. The parts involving two and four variables are obviously the same in $\Delta_a \Delta_b$ and in $\Delta_b \Delta_a$. The part involving three variables, $x_a$, $x_b$, $x_j$ is more tricky, it is:
 \begin{equation}
 \label{eq_x9}
  \frac{1-(a,j)}{x_a-x_j} \cdot  \frac{1-(b,j)}{x_b-x_j} +   \frac{1-(a,b)}{x_a-x_b} \cdot  \frac{1-(b,j)}{x_b-x_j} +   \frac{1-(a,j)}{x_a-x_j} \cdot  \frac{1-(b,a)}{x_b-x_a},
 \end{equation}
 where in each fraction we first apply the operator in numerator and then divide by the denominator. The desired relation \eqref{eq_Delta_commute} follows from the symmetry of \eqref{eq_x9} in $a\leftrightarrow b$, which is verified directly by moving all permutation operators to the right.
 The next useful identity, which is verified directly, is that for $a\ne b$
 \begin{equation}
 \label{eq_x14}
  [x_a,\Delta_b]=x_a \Delta_b- \Delta_b x_a=(a,b). % & \text{ if }a\ne b,\\ -\sum_{j\ne a} (a,j), &\text{ if } a=b.\end{cases}
 \end{equation}
 Another direct verification yields that for $a\ne b$ we have
 \begin{equation} \label{eq_x15}
  \left[\frac{\partial}{\partial x_a},\Delta_b\right]=\frac{1-(a,b)}{(x_a-x_b)^2} +\frac{1}{x_a-x_b} \left(\frac{\partial}{\partial x_a}-\frac{\partial}{\partial x_b}\right) (a,b).
 \end{equation}
 The last identity makes it clear that $\left[\tfrac{\partial}{\partial x_a},\Delta_b\right]= \left[\tfrac{\partial}{\partial x_b},\Delta_a\right]$. Combining with \eqref{eq_Delta_commute}, this implies commutativity of the Dunkl operators: $\D_a \D_b=\D_b \D_a$. Next, for each $ a \ne b $ we have
\begin{equation}
 \label{eq_commutator}
  [x_a \D_a, x_b \D_b]=  \theta (x_a \D_a - x_b \D_b) \cdot (a,b).
\end{equation}
Indeed, using \eqref{eq_x14} and commutativity of $x_a$ with $\frac{\partial}{\partial x_b}$, we have
$$
 x_a \D_a x_b \D_b - x_b \D_b x_a \D_a= x_a x_b \D_a \D_b - x_a (a,b) \D_b - x_b x_a \D_b \D_a +\theta x_b (a,b) \D_a.
$$
Cancelling the first and third terms and using $(a,b) \cdot  \D_b \cdot (a,b)=\D_a$ we arrive at \eqref{eq_commutator}. Using this identity we further compute
\begin{multline}
[x_a \D_a, (x_b \D_b)^m]=\sum_{k=0}^{m-1} (x_b \D_b)^k [x_a \D_a,x_b \D_b] (x_b\D_b)^{m-k-1}\\=\theta \sum_{k=0}^{m-1} (x_b \D_b)^k (x_a \D_a - x_b \D_b) \cdot (a,b) (x_b\D_b)^{m-k-1}= \\ \theta \sum_{k=0}^{m-1} (x_b \D_b)^k (x_a \D_a - x_b \D_b) \cdot (x_a\D_a)^{m-k-1} \cdot  (a,b)=
\theta \bigl((x_a \D_a)^m - (x_b \D_b)^m\bigr) \cdot (a,b).
\end{multline}
Iterating the same argument again, we get
\begin{multline*}
[(x_a \D_a)^l, (x_b \D_b)^m]=\sum_{k=0}^{l-1} (x_a \D_a)^k [x_a \D_a,(x_b \D_b)^m] (x_a\D_a)^{l-k-1}\\
%=\theta \sum_{k=0}^{l-1} (x_a \D_a)^k  \bigl((x_a \D_a)^m - (x_b \D_b)^m\bigr) \cdot (a,b) (x_a\D_a)^{\ell-k-1}\\
=\theta  \left(  \sum_{k=m}^{m+l-1} (x_a \D_a)^{m+k} (x_b\D_b)^{m+l-k-1}  - \sum_{k=0}^{l-1} (x_a \D_a)^k (x_b \D_b)^{m+l-k-1}\right)  \cdot (a,b).
\end{multline*}
Interchanging $l$ and $m$ and using the skew-symmetry $[f,g]=-[g,f]$, we also have
\begin{multline*}
[(x_b \D_b)^l, (x_a \D_a)^m]=\theta  \left( \sum_{k=0}^{m-1} (x_a \D_a)^k (x_b \D_b)^{m+l-k-1}- \sum_{k=l}^{m+l-1} (x_a \D_a)^{l+k} (x_b\D_b)^{m+l-k-1}\right)  \cdot (a,b).
\end{multline*}
Summing the last two identities over all $a<b$, we arrive at
\begin{multline}
 \left[\sum_{a=1}^N(x_a \D_a)^l, \sum_{b=1}^N(x_b \D_b)^m\right]\\ =\theta \sum_{a<b} \left(\sum_{k=0}^{m-1} + \sum_{k=m}^{l+m-1} -\sum_{k=0}^{l-1} -\sum_{k=l}^{m+l-1} \right) (x_a \D_a)^k (x_b \D_b)^{m+l-1-k} (a,b)=0.\qedhere
\end{multline}
\end{proof}
\begin{lemma}\label{Lemma_CMS}
 The restriction of the operator $\P_2$ on the space of symmetric polynomials coincides with
 \begin{equation} \label{eq_CMS}
   \sum_{i=1}^N \left(x_i \frac{\partial}{\partial x_i}\right)^2 + \theta \sum_{i\ne j} \frac{ x_i(x_i+x_j)}{x_i-x_j} \frac{\partial}{\partial x_i}.
 \end{equation}
\end{lemma}
\begin{remark} Changing the variables  $x_a=\exp(\ii z_a)$, conjugating with $\prod_{a<b} \sin^{\theta}\left( \frac{z_a-z_b}{2}\right)$ and shifting by a constant, \eqref{eq_CMS} turns into $-H$, where $H$ is the trigonometric version of the so-called Calogero-Moser-Sutherland Hamiltonian:
$$
 H=\sum_{i=1}^N \left(\frac{\partial}{\partial z_i}\right)^2 + \frac{\theta(\theta-1)}{2}\sum_{i<j}  \frac{1}{\sin^2 \left(\frac{z_a-z_b}{2}\right)}.
$$

\end{remark}
\begin{proof}[Proof of Lemma \ref{Lemma_CMS}] Since $1-(i,j)=0$ on symmetric polynomials, $(x_i \D_i)^2$ restricts as
$$
 \left(x_i \frac{\partial}{\partial x_i} + \theta x_i \sum_{j\ne i} \frac{1-(i,j)}{x_i-x_j}\right) \left(x_i \frac{\partial}{\partial x_i}\right)=
 \left(x_i \frac{\partial}{\partial x_i}\right)^2+\theta \sum_{ j\ne i} \frac{x_i}{x_i-x_j} \left[ x_i \frac{\partial}{\partial x_i} - x_j \frac{\partial}{\partial x_j} (i,j)\right].
$$
The permutation $(i,j)$ acts identically on symmetric polynomials and can be removed from the right-hand side of the last formula. Summing over all $i$, we arrive at \eqref{eq_CMS}.
\end{proof}
\begin{lemma} \label{Lemma_Symmetric_is_Jack}
 For identical representation $\tau=(N)$ the joint eigenfunction $F_{\lambda,(N),1}$ of operators $\P_m$ from Theorem \ref{Theorem_classification_of_ef} is (up to a constant factor) the symmetric Jack polynomial $J_\lambda(x_1,\dots,x_N;\, \theta)$.
\end{lemma}
\begin{proof}
 As in the proof of Theorem \ref{Theorem_classification_of_ef}, we note that the linear space $\bigoplus\limits_{\mu\mid \mu\preceq \lambda} V_{\mu}$
 is invariant  both for $\P_m$ and for the symmetric group $S_N$. Because $S_N$-action commutes with $\P_m$ action, $\P_m$ preserves the subspace of symmetric polynomial inside $\bigoplus\limits_{\mu\mid \mu\preceq \lambda} V_{\mu}$. Hence, $\P_m$ must have some symmetric eigenfunctions in  $\bigoplus\limits_{\mu\mid \mu\preceq \lambda} V_{\mu}$ and then by triangularity consideration, the unique symmetric eigenfunction with leading monomial in $V_\lambda$ has to be $F_{\lambda,(N),1}$. We conclude that $F_{\lambda,(N),1}$ is symmetric. Since $\P_m$ is self-adjoint by Lemma \ref{Lemma_self_adjoint}, its eigenfunctions are orthogonal with respect to scalar product of that lemma. We conclude that as $\lambda$'s vary, the polynomials $F_{\lambda,(N),1}$ form an orthogonal basis of the space $\Lambda_N$ of symmetric polynomials in $N$ variables and this basis is related to monomial symmetric polynomials $m_\lambda$ by a triangular transformation. Being such an orthogonal basis is a property which uniquely determines Jack polynomials and can be taken as one of their equivalent definitions, see \cite{stanley1989some} or \cite[Section 6.10]{macdonald1998symmetric} and note that they use $\alpha=1/\theta$ as the parameter.
\end{proof}

\bibliographystyle{alpha}
\bibliography{P_eigenvalues}

\end{document}